\documentclass[oneside,a4paper]{article}
\usepackage{amssymb,amsmath,t1enc,amsthm,geometry,enumitem}
\usepackage[colorlinks]{hyperref}
\usepackage[utf8]{inputenc} 
\usepackage[english]{babel} 
\usepackage{color}

\newcommand{\halmazvonal}[2]{\left\{\,#1\mid #2\,\right\}}

\newcommand{\Hol}{\mathcal{H}ol}
\newcommand{\hol}{\mathfrak{hol}\hspace{1pt}}
\newcommand{\ihol}{\mathfrak{hol}^*}

\newcommand*{\T}{\widehat {T}}

\newcommand{\R}{\mathbb R}
\newcommand*{\F}{\mathcal{F}}
\newcommand*{\I}{\mathcal{I}}
\newcommand*{\pols}{\mathcal Pol (\mathbb S^{n-1})}

\newcommand{\norm}[1]{\ensuremath{\left\lvert #1 \right\rvert}}
\newcommand*{\X}[1]{{\mathfrak X}(#1)}
\newcommand{\diff}[1]{\mathcal{D}i\!f \hspace{-3pt}f(#1)}
\newcommand{\diffp}[1]{\mathcal{D}i\!f \hspace{-3pt} f_+(#1)}
\newcommand{\diffo}[1]{\mathcal{D}i\!f \hspace{-3pt} f_o(#1)}

\def\P{\mathcal P}
\def\xo{x_{\! o}}

\theoremstyle{plain}
\newtheorem{theorem}{Theorem}[section]
\newtheorem{proposition}[theorem]{Proposition}
\newtheorem{lemma}[theorem]{Lemma}

\theoremstyle{definition}
\newtheorem{definition}[theorem]{Definition}
 
\newtheorem{remark}[theorem]{Remark} 
\newtheorem{example}{Example}

\hypersetup{
    colorlinks = true,
    linkcolor=blue,
}

\title{The holonomy of spherically symmetric projective Finsler metrics of
  constant curvature}

\author{Mezrag, Asma
  \\
  \small \texttt{asma.mezrag@science.unideb.hu}
\\
  \small University of Debrecen,
  \\
  \small Institute of Mathematics,
  \\
  \small Pf.~400, Debrecen, 4002,
  \\
  \small Hungary
  \and
  Muzsnay, Zoltan
  \\
  \small \texttt{muzsnay@science.unideb.hu}
  \\
  \small University of Debrecen,
  \\
  \small Institute of Mathematics,
  \\
  \small Pf.~400, Debrecen, 4002,
  \\
  \small Hungary }

\begin{document}

\maketitle

\begin{abstract}
  In this paper, we investigate the holonomy group of $n$-dimensional
  projective Finsler metrics of constant curvature. We establish that in
  the spherically symmetric case, the holonomy group is maximal, and for a
  simply connected manifold it is isomorphic to $\diffo{\mathbb S^{n-1}}$,
  the connected component of the identity of the group of smooth
  diffeomorphism on the ${n-1}$-dimensional sphere. In particular, the
  holonomy group of the $n$-dimensional standard Funk metric and the
  Bryant--Shen metrics are maximal and isomorphic to
  $\diffo{\mathbb S^{n-1}}$. These results are the firsts describing
  explicitly the holonomy group of $n$-dimensional Finsler manifolds in the
  non-Berwaldian (that is when the canonical connection is non-linear)
  case.
\end{abstract}

\bigskip \bigskip

  
\medskip
  

\bigskip \bigskip

\section{Introduction}

The holonomy group of a Riemannian or Finslerian manifold is a very natural
algebraic object attached to the geometric structure: it is the group
generated by parallel translations along loops with respect to the
canonical connection. Riemannian holonomy groups have been extensively
studied and their complete classification is now known.

On the Finslerian holonomy, relatively few results are known.  Z.I.~Szabó
proved in \cite{Szabo_1981} that in the case where the Finslerian parallel
translation is linear, then there exist Riemannian metrics having the same
holonomy group, and for Landsberg metrics L.~Kozma showed in
\cite{Kozma_1996} that the holonomy groups are compact Lie groups
consisting of isometries of the indicatrix with respect to an induced
Riemannian metric.  The first paper showing that the holonomy property of
Finsler manifolds can be very different from the Riemannian is
\cite{Muzsnay_Nagy_2012_hjm} by proving that the holonomy group of a
Finsler manifold is not necessarily a compact Lie group.  In
\cite{Muzsnay_Nagy_max_2015} it was proven that there are $2$-dimensional
Finsler surfaces with infinite dimensional holonomy group, isomorphic to
$\diffp{\mathbb S^{1}}$, the orientation preserving diffeomorphism group of
the circle in the orientable case, and as a consequence, isomorphic to
$\diff{\mathbb S^{1}}$, the diffeomorphism group of the circle in the
non-orientable case.  We emphasize that explicit examples of Finsler
holonomy groups for nonlinear connection existed only in the
$2$-dimensional cases.

In this article we are focusing on the holonomy structure of
$n$-dimensional Finsler manifolds.  We use the method developed in
\cite{Hubicska_Muzsnay_2018b} for the investigation of holonomy properties
of Finsler manifolds by constructing a tangent Lie algebra to the holonomy
group, called the holonomy algebra, and its subalgebra: the infinitesimal
holonomy algebra.  After a brief introduction of the basic notation in
Section \ref{sec:2}, we give the most important element of this theory in
Chapter \ref{sec:3}. We give a sufficient condition on the holonomy group
to be maximal in terms of the infinitesimal holonomy algebra in Theorem
\ref{thm:1_ihol}.

In Chapter \ref{sec:4} we investigate the holonomy structure of projective
Finsler manifold $(M, F)$ of constant curvature.  In \cite{Shen_2003b},
Z.~Shen proved that in the $x$-analytic case, those metrics are determined
by their Finsler norm function and projective factor at a single point
$\xo \in M$. We consider the \emph{spherically symmetric} case, where there
exists a point $\xo \in M$ where both the norm and the projective factor are
nonzero multiples of the Euclidean norm. If the curvature vanishes, then
the horizontal distribution associated with the canonical connection in the
tangent bundle is integrable, and hence the holonomy group is trivial. This
is why we are focusing on the non-flat case.  In Theorem
\ref{thm:2_hol_group_max} we show that on a simply connected manifold, a
spherically symmetric projective Finsler metrics of nonzero constant
curvatures have maximal holonomy group.  The holonomy groups in those cases
are isomorphic to $\diffo{\mathbb S^{n-1}}$, the connected component of the
identity of the group of smooth diffeomorphism on the ${n-1}$-dimensional
sphere.  In particular, we obtain that the holonomy group of the
$n$-dimensional standard Funk metric and the $n$-dimensional Bryant--Shen
metrics are isomorphic to $\diffo{\mathbb S^{n-1}}$.

\bigskip

\section{Preliminaries}
\label{sec:2}

Throughout this article, $M$ is a $C^\infty$ smooth manifold, $\X{M}$ is
the vector space of smooth vector fields on $M$ and $\diff{M}$ is the group
of all $C^\infty$-diffeomorphism of $M$.  The first and the second tangent
bundles of $M$ are denoted by $(TM,\pi ,M)$ and $(TTM,\tau ,TM)$,
respectively.  Local coordinates $(x^i)$ on $M$ induce local coordinates
$(x^i, y^i)$ on $TM$.  \medskip

\subsection{Finsler manifolds, canonical connection}
\label{sec:finsl}   

A \emph{Finsler manifold} is a pair $(M,\mathcal F)$, where the norm
function $F\colon TM \to \mathbb{R}_+$ is continuous, smooth on
$\T M \!:= \!TM\!  \setminus\! \{0\}$, its restriction $F_x=F|_{_{T_xM}}$
is a positively homogeneous function of degree one and the symmetric
bilinear form
\begin{displaymath}
  g_{x,y} \colon (u,v)\ \mapsto \ g_{ij}(x, y)u^iv^j=\frac{1}{2}
  \frac{\partial^2 F^2_x(y+su+tv)}{\partial s\,\partial t}
  \Big|_{t=s=0}
\end{displaymath}
is positive definite at every $y\in \hat T_xM$. The hypersurface of $T_xM$
defined by
\begin{equation}
  \label{eq:indicatrix}
  \I_x  = \halmazvonal{y \in T_xM}{F(x, y) \! = \! 1},
\end{equation}
is called the \emph{indicatrix} at $x \in M$. We note that at any point
$x\in M$ the indicatrix is diffeomorphic to the $(n-1)$-dimensional sphere.
\emph{Geodesics} of $(M, F)$ are determined by a system of $2$nd order
ordinary differential equation
\begin{equation}
  \label{eq:geodesic}
  \ddot{x}^i + 2 G^i(x,\dot x)=0, \quad i = 1,\dots,n,
\end{equation}
in a local coordinate system $(x^i,y^i)$ of $TM$, where
$G^i(x,y)$ are given by
\begin{equation}
  \label{eq:G_i}
  G^i(x,y):= \frac{1}{4}g^{il}(x,y)\Big(2\frac{\partial
    g_{jl}}{\partial x^k}(x,y) -\frac{\partial g_{jk}}{\partial
    x^l}(x,y) \Big) y^jy^k.
\end{equation}
A vector field $X(t)=X^i(t)\frac{\partial}{\partial x^i}$ along a curve
$c(t)$ is said to be parallel with respect to the associated
\emph{homogeneous (nonlinear) connection} if it satisfies
\begin{equation}
  \label{eq:D}
  D_{\dot c} X (t):=\Big(\frac{d X^i(t)}{d t}
  + G^i_j(c(t),X(t))\dot c^j(t)\Big)\frac{\partial}{\partial x^i}=0, 
\end{equation}
where $ G^i_j=\frac{\partial G^i}{\partial y^j}$.
\\[1ex]
The \emph{horizontal Berwald covariant derivative} $\nabla_X\xi$ of
$\xi(x,y) = \xi^i(x,y)\frac {\partial}{\partial y^i}$ by the vector field
$X(x) = X^i(x)\frac {\partial}{\partial x^i}$ is expressed locally by
\begin{equation}
  \label{covder}
  \nabla_X\xi = \left(\frac {\partial\xi^i(x,y)}{\partial x^j} 
    - G_j^k(x,y)\frac{\partial \xi^i(x,y)}{\partial y^k} + 
    G^i_{j k}(x,y)\xi^k(x,y)\right)X^j\frac {\partial}{\partial y^i}, 
\end{equation}
where we denote
$G^i_{j k}(x,y) := \frac{\partial G_j^i(x,y)}{\partial y^k}$. In the sequel
we will use the simplified notation
\begin{math}
  \nabla_k \xi = \nabla_{\frac{\partial}{\partial x^k}}\xi
\end{math}
for the horizontal Berwald covariant derivatives with respect to the
coordinate directions. 
\\[1ex]
The \emph{Riemannian curvature tensor} field
\begin{math}
  R_{}\!= \! R^i_{jk}(x,y) dx^j\otimes dx^k \otimes
  \frac{\partial}{\partial x^i}
\end{math}
has the expression
\begin{displaymath}
  R^i_{jk}(x,y) =  \frac{\partial G^i_j(x,y)}{\partial x^k} 
  - \frac{\partial G^i_k(x,y)}{\partial x^j} + 
  G_j^m(x,y)G^i_{k m}(x,y) - G_k^m(x,y)G^i_{j m}(x,y). 
\end{displaymath}

 \bigskip

 \subsection{Projective Finsler manifold with constant curvature}
\label{sec:2_2}

A Finsler function $F$ on an open subset $D \subset \mathbb R^n$ is said to
be \emph{projective} or projectively flat, if all geodesic curves are
straight lines in $D$. A Finsler manifold is said to be \emph{locally
  projective} or locally projectively flat, if at any point there is a
local coordinate system $(x^i)$ in which $F$ is projective.

Let $(x^1,\dots ,x^n)$ be a local coordinate system on $M$ corresponding to
the canonical coordinates of the Euclidean space which is projectively
related to $(M, F)$. Then the geodesic coefficients \eqref{eq:G_i} and
their derivatives have the form
\begin{equation}
  \label{eq:proj_flat_G_i}
  G^i = \P(x,y)y^i, \qquad  G^i_k = \frac{\partial
    \P}{\partial y^k}y^i + \P\delta^i_k,\qquad G^i_{kl} = \frac{\partial^2
    \P}{\partial y^k\partial y^l}y^i + \frac{\partial \P}{\partial
    y^k}\delta^i_l + \frac{\partial \P}{\partial y^l}\delta^i_k,
\end{equation}
where $\P$ is a 1-homogeneous function in $y$, called the projective factor
of $(M,F)$. According to \cite[Lemma 8.2.1, p.~155]{Chern_Shen_2005} the
projective factor can be computed using the formula
\begin{equation}
  \label{eq:P}
  \P(x,y) = \frac{1}{2\F} \frac{\partial \F}{\partial x^i}y^i. 
\end{equation}
The Finsler manifold has \emph{constant flag curvature}
$\lambda\in{\mathbb R}$, if for any $x\in M$ the local expression of the
Riemannian curvature is
\begin{equation}
\label{eq:curv_tensor}
  R^i_{jk}(x,y) = \lambda\big(\delta_k^ig_{jm}(x,y)y^m -
  \delta_j^ig_{km}(x,y)y^m\big).
\end{equation}
In this case the flag curvature of the Finsler manifold
(cf.~\cite{Chern_Shen_2005}, Section 2.1) does not depend on the point, nor
on the 2-flag.

\begin{lemma}[\cite{Muzsnay_Nagy_2015}]
  \label{lemma:cov_R}
  The horizontal covariant derivative $\nabla_WR$ of the tensor field
  $R = R^i_{jk}(x,y)dx^j\wedge dx^k\frac{\partial}{\partial x^i}$ vanishes.
\end{lemma}
\begin{proof}
  The Lemma is a consequence of the fact, that the horizontal covariant
  derivative of the Finsler function vanishes.  Indeed, \cite[Lemma 6.2.2,
  p.~85]{Shen_2001} yields
  \begin{displaymath}
    \nabla_wg_{(x,y)}(u,v) = -2 L(u,v,w),
  \end{displaymath}
  for any $u,v,w\in T_xM$, where $L$ is the Landsberg curvature of the Finsler
  metric $\F$ \cite[Chapter 6.2]{Shen_2001}. Moreover $\nabla_W y = 0$,\;
  $\nabla_W \text{Id}_{TM} = 0$ for any vector field
  $W\in{\mathfrak X}^{\infty}(M)$, and $L_{(x,y)}(y,v,w)= 0$ (cf. equation
  6.28, p. 85 in \cite{Shen_2001}). Hence we obtain $\nabla_WR = 0$.

\end{proof}

\bigskip

\section{Holonomy}
\label{sec:3} 

\subsection{Parallel translation and the holonomy group }
\label{sec:finsler2}   

Let $(M, \F)$ be a Finsler manifold. The \emph{parallel translation}
$\tau_{c} \colon T_{c(0)}M \rightarrow T_{c(1)}M$ along a curve $c:[0,1]\to \R$ is
defined by vector fields $X(t)$ along $c(t)$ which are solutions of the
differential equation \eqref{eq:D}. Since $\tau_{c}:T_{c(0)}M\to T_{c(1)}M$ is
a differentiable map between $\hat T_{c(0)}M$ and $\hat T_{c(1)}M$ preserving
the value of the Finsler norm, it induces a map
\begin{equation}
  \label{eq:parallel_3}
  \tau_{c}\colon \I_{c(0)} \longrightarrow \I_{c(1)},
\end{equation}
between the indicatrices. 

The \emph{holonomy group} $\Hol_x(M,\F)$ of a Finsler manifold $(M, \F)$ at a
point $x\in M$ is the group generated by parallel translations along
piece-wise differentiable closed curves starting and ending at $x$. Since the
parallel translation \eqref{eq:parallel_3} is 1-homogeneous and preserves the
norm, one can consider it as a map on the indicatrices
\begin{equation}
  \label{eq:hol_elem}
  \mathcal T_{c}:\I_x \to \I_x,
\end{equation}
therefore, the holonomy group can be seen as a subgroup of the
diffeomorphism group of the indicatrix:
\begin{equation}
  \label{eq:hol_group}
  \Hol_x (\F) \ \subset \ \diff{\I_x}. 
\end{equation}

\bigskip

\subsection{Holonomy algebra and infinitesimal holonomy algebra}

The tangent Lie algebra (see \cite{Hubicska_Muzsnay_2018b}) of the holonomy
group $\Hol_x (\F)$ is called the \emph{holonomy algebra} and is denoted
as $\hol_x (\F)$. The holonomy algebra can give information about the
holonomy property of the Finsler manifold.

When $\Hol_x (\F)$ is a finite-dimensional Lie group, $\hol_x (\F)$ is its
Lie algebra. In particular, for a Riemannian metric, $\Hol_x (\F)$ is a Lie
subgroup of the orthogonal group \cite{Borel_Lichnerowicz_1952}, and
$\hol_x (\F)$ is its Lie algebra. In the Finslerian case, however, it may
happen that $\Hol_x (\F)$ is not a finite dimensional Lie group
\cite{Muzsnay_Nagy_2014, Muzsnay_Nagy_max_2015, Muzsnay_Nagy_2015}.

Considering the tangent spaces of both sides in \eqref{eq:hol_group} we
obtain
\begin{equation}
  \label{eq:hol_x}
  \hol_x (\F) \ \subset \  \X{\I_x}.
\end{equation}
The most important properties of the holonomy algebra are given by the
following

\medskip

\begin{proposition}[\cite{Hubicska_Muzsnay_2018b}]
  \label{prop:hol} \ Let $(M, \F)$ be a Finsler manifold. Then the holonomy
  algebra $\hol_x (\F)$ is a Lie subalgebra of $\X{\I_x}$, and its
  exponential image is in the topological closure of the holonomy group,
  that is
  \begin{equation}
    \label{eq:exp_hol}
    \exp \bigl(\hol_x (\F)\bigr) \subset \overline{\Hol_x(\F)},
  \end{equation}
  where the overline denotes the topological closure of the holonomy group
  with respect to the $C^\infty$--topology of $\diff{\I_x}$.
\end{proposition}
From Proposition \ref{prop:hol} one can obtain, that a Lie subalgebra of
$\X{\I_x}$ generated by any subset of $\hol_x (\F)$ is also a Lie
subalgebra of $\hol_x (\F)$, and in particular, its elements have the
tangent property to the holonomy group $\Hol_x(\F)$.

One can show that for any point $x\in M$ and any tangent vectors
$X_1,X_2 \in T_xM$, the curvature vector field $R(X_1, X_2) \in \X{\I_x}$
considered as
\begin{equation}
  \label{eq:curv_fields}
  y \to R_{(x,y)}(X_1, X_2),
\end{equation}
and its successive convariant derivatives are tangent to the holonomy group
\cite{Hubicska_Muzsnay_2018b}. It follows that the Lie subalgebra of vector
fields on the indicatrix $\I_x$ generated by the curvature vector fields
\eqref{eq:curv_fields} and their successive covariant derivatives
\begin{equation}
  \label{eq:inf_hol_alg}
  \ihol_x(\F):=
  \left\langle
    \nabla_{X_k} \dots \nabla_{X_3} R(X_1, X_2) \ \big| \ X_1, \dots,
    X_k\in \X{M}  \right\rangle_{Lie},
\end{equation}
is a Lie subalgebra of $\hol_x (\F)$. From \eqref{eq:exp_hol} we get
\begin{equation}
  \label{eq:exp_ihol}
  \exp \bigl(\ihol_x (\F)\bigr) \subset \overline{\Hol_x(\F)}.
\end{equation}
\begin{definition}
  The Lie algebra $\ihol_x(\F)$ defined in \eqref{eq:inf_hol_alg} is called
  the \emph{infinitesimal holonomy algebra} of the Finsler space $(M,\F)$
  at $x \in M$.
\end{definition}
\noindent
One has the inclusion of Lie algebras:
\begin{equation}
  \label{eq:curv_hol_x}
  \mathfrak{hol}_x^{*}(\F) \subset  \hol_x (\F) \subset \X{\I_x}, 
\end{equation}
therefore, at the level of groups, we get
\begin{equation}
  \label{eq:group_curv_hol_x}
  \exp \bigl(\ihol_x(\F) 
  \bigr) \subset   \exp \bigl(\hol_x (\F)\bigr)
  \subset \overline{\Hol_x(\F)}  \subset \diff{\I_x}.
\end{equation}

\smallskip

We give a sufficient condition on the holonomy group to be maximal in terms
of the infinitesimal holonomy algebra in the following

\begin{theorem}
  \label{thm:1_ihol}
  Let $(M, \F)$ be an $n$-dimensional simply connected Finsler manifold.
  If the infinitesimal holonomy algebra $\ihol_x(\F)$ at some point
  $x\in M$ is dense in the Lie algebra $\X{\I_x}$ of the vector fields on
  the indicatrix $\I_x$, then the holonomy group is maxmial:
  \begin{equation}
    \label{eq:closure_Hol}
    \overline{\Hol_x(\F)} \cong \diffo{\mathbb S^{n-1}},
  \end{equation}
  that is its closure is isomorphic to the connected component of the
  identity in the diffeomorphism group of the $n-1$-dimensional sphere.
\end{theorem}

\begin{proof}
  Let $M$ be a simply connected $n$-dimensional manifold. At any point
  $x\in M$ we have \eqref{eq:hol_group}. Since $M$ is simply connected, any
  closed curve can be shrunk to a point, therefore
  \begin{equation}
    \label{eq:hol_group_simply}
    \Hol_x (\F) \ \subset \ \diffo{\I_x}.
  \end{equation}
  On the other hand, from the hypotheses, the infinitesimal holonomy
  algebra $\ihol_x(\F)$ is dense in $\X{\I_x}$, therfore its closure
  satisfies
  \begin{equation}
    \label{eq:A_closure}
    \overline{\ihol_x(\F)} = \X{\I_x}.
  \end{equation}
  Since the exponential mapping is continuous (c.f.~Lemma 4.1 in
  \cite{Omori_1997}, p.~79), we have 
  \begin{equation}
    \label{eq:cont_exp}
    \exp\big(\overline{\ihol_x(\F) }\big) \subset
    \overline{\exp(\ihol_x(\F))},
  \end{equation}
  hence taking into account \eqref{eq:A_closure}, \eqref{eq:cont_exp},
  \eqref{eq:exp_ihol}, and \eqref{eq:hol_group_simply} respectively, we have
  \begin{equation}
    \label{eq:F_in_Dif_1}
    \exp(\X{\I_x})=\exp\big(\overline{\ihol_x(\F) }\big) \subset
    \overline{\exp(\ihol_x(\F))} \subset \overline{\Hol_x(\F)}
    \subset \diffo{\I_x},
  \end{equation}
  which gives for the generated groups the following relations
  \begin{equation}
    \label{eq:F_in_Dif_2}
    \big\langle\! \exp(\X{\I_x}) \!\big\rangle_{\mathrm{group}} 
    \subset \overline{\Hol_x(\F)}
    \subset \diffo{\I_x}.
  \end{equation}
  Moreover, the conjugation map
  \begin{math}
    Ad:\diffo{\I_x} \times \X{\I_x} \longrightarrow \X{\I_x}
  \end{math}
  satisfies the relation
  \begin{displaymath}
    h (\exp s\xi)  \,h^{-1} =\exp s\, Ad_h \xi,
  \end{displaymath}
  for every $h\in\diffo{\I_x}$ and $\xi\in\X{\I_x}$.
  Since the Lie algebra $\X{\mathbb \I_x}$ is invariant under conjugation,
  therefore the group
  \begin{math}
    \big\langle\!  \exp(\X{\I_x}) \!\big\rangle_{\mathrm{group}}
  \end{math}
  is also invariant under conjugation and consequently, it is a non-trivial
  normal subgroup of $\diffo{\I_x}$.  From \cite[Theorem 1.]{Thurston_1974}
  we know that $\diffo{\I_x}$ is a simple group, its only non-trivial
  normal subgroup is itself, we get that
  \begin{displaymath}
    \big\langle\!  \exp{\X{\I_x}} \!\big\rangle_{\mathrm{group}} =\diffo{\I_x},
  \end{displaymath}
  and \eqref{eq:F_in_Dif_2} reads as
  \begin{equation}
    \label{eq:F_in_Dif_2_2}
    \diffo{\I_x}
    \subset \overline{\Hol_x(\F)}
    \subset \diffo{\I_x}, 
  \end{equation}
  that is 
  \begin{equation}
    \label{eq:Hol_I}
    \overline{\Hol_x(\F)} = \diffo{\I_x}.
  \end{equation}
  Since $\I_x$ and $\mathbb S^{n-1}$ are diffeomorphic, their
  diffeomorphism groups and the connected component of their diffeomorphism
  groups are isomorphic, therefore from \eqref{eq:Hol_I} we obtain 
  \begin{equation}
    \label{eq:Hol_S}
    \overline{\Hol_x(\F)}\cong \diffo{\mathbb S^{n-1}} .
  \end{equation}
\end{proof}

\bigskip

\section{The holonomy of spherically symmetric projective Finsler metrics
  with constant flag curvature}
\label{sec:4}

Z.~Shen in his paper \cite{Shen_2003b} investigated projective Finsler
metrics with constant flag curvature, and give a complete classification in
the $x$-analytical case. In particular, he showed that a projective Finsler
metrics $\F$ with constant flag curvature is completely determined -- using
an $\xo$ centered coordinate system -- by $\F(0,y)$ and $\P(0,y)$, where
$\P$ denotes the projective factor.  In this chapter we investigate the
case when those data are spherically symmetrical, that is at some
particular point $\xo \in M$, the Finsler function and the projective factor
are both a multiple of the Euclidean norm
\begin{equation}
  \label{eq:multiple_norm_2}
  \F(\xo, y)=c_1 \norm{y}, \quad 
  \P(\xo, y)=c_2 \norm{y},
\end{equation}
with $c_1, c_2 \neq 0$.
\begin{remark}
  \label{rem:sphericallysymmetric}
  Without loss of generality, the constant $c_1$ may be considered equal to
  1, and instead of \eqref{eq:multiple_norm_2} to have at $\xo \in M$ the
  relations:
  \begin{equation}
    \label{eq:multiple_norm}
  \F(\xo, y)=\norm{y}, \qquad   \P(\xo, y)=c \cdot \norm{y}, 
\end{equation}
where $c\neq 0$. Indeed, there is no loss of generality, since replacing
the Finsler norm function $\F$ by a positive constant multiple of it does
not change the geodesic equation \eqref{eq:geodesic}, the parallelism
\eqref{eq:D}, therefore it does not change the holonomy sturcture.
\end{remark}
We note that if \eqref{eq:multiple_norm} is satisfied, then the indicatrix
at $\xo$ is
\begin{equation}
  \label{eq:I_x0}
  \I_{\xo}= \bigl\{ y\in T_{\xo}M \ \big| \ \norm{y} \! = \! 1
  \bigr\} =  \mathbb S^{n-1} \quad \subset \R^n,
\end{equation}
the $n-1$ dimensional Euclidean sphere.  Moreover, for the geodesic
coefficients \eqref{eq:proj_flat_G_i} at $\xo$ we have
\begin{equation}
  \label{eq:proj_flat_G_ijk}
  \begin{aligned}
    G^i(\xo,y) & = c \norm{y} y^i,
    \\
    G^i_j (\xo,y) &
    %
    %
    = c \left( \frac{y^iy^j}{\norm{y}} + \norm{y}\delta^i_j \right),
    \\
    G^i_{jk}(\xo,y) &
    = c 
    \left(
       \frac{y^i}{\norm{y}} \delta^j_k + \frac{y^j}{\norm{y}} \delta^i_k +
      \frac{y^k}{\norm{y}} \delta^i_j - \frac{y^iy^jy^k}{\norm{y}^3}
    \right).
  \end{aligned}
\end{equation}
Using \cite[Lemma 8.2.1]{Chern_Shen_2005} we get for a projective metric:
\begin{equation}
  \label{eq:P_x}
  \frac{\partial \P}{\partial x^m} = \P \frac{\partial \P}{\partial y^m}
  -\lambda \F \frac{\partial \F}{\partial y^m}
  = \tfrac{1}{2}  \frac{\partial (\P^2 -  \lambda \F^2)}{\partial y^m},  
\end{equation}
therefore at $\xo$ where we have \eqref{eq:multiple_norm} we get
\begin{equation}
  \label{eq:P_x_0}
  \begin{aligned}
    \frac{\partial \P}{\partial x^m}(\xo,y) = (c^2 -\lambda) y^m, \qquad
  \end{aligned}
\end{equation}
and  at $\xo$ we have 
\begin{equation}
  \label{eq:proj_flat_G_i_x_m}
  \begin{aligned}
    \frac{\partial G^i_k}{\partial x^m}(\xo,y) = (c^2 -\lambda) ( y^i
    \delta^m_k + y^m \delta^i_k ).
    %
  \end{aligned}
\end{equation}
From \eqref{eq:curv_tensor} we get that at $\xo$ the coefficients of the
curvature tensor are
\begin{equation}
  \label{eq:curv}
  R^l_{ij}(\xo,y) = \lambda
  \bigl(
  \delta_j^l\delta^i_my^m - \delta_i^l\delta^j_m y^m 
  \bigr)
  = \lambda\bigl(\delta_j^l y^i -\delta_i^l y^j
  \bigr),
\end{equation}
therefore the curvature vector fields are
\begin{equation}
  \label{eq:curv_field}
  \xi_{ij}=R_{ij}^s \tfrac{\partial}{\partial y^s}
  = \lambda\left( y^i \frac{\partial}{\partial y^j} - y^j
    \frac{\partial}{\partial y^i}
  \right),
\end{equation}
the infinitesimal generators of rotations.

\bigskip

\noindent
Let us introduce the multiindex notation: for
$\mathbf{m}:=(m_1, \dots m_n)$ its length is
$\ell(\mathbf{m})=m_1 + \dots +m_n$, and
\begin{math}
  \mathbf{y}^\mathbf{m}=\prod_{i=1}^n (y^i)^{m_i}= (y^1)^{m_1}\dots
  (y^n)^{m_n}.
\end{math}
We define
\begin{equation}
  \label{eq:4_p}
  {\mathcal A}_p= Span_{_\R}\left\{
    \frac{\mathbf{y}^{\mathbf{m}}}{\norm{\mathbf{y}}^{\ell({\mathbf{m}})}} \,
    \xi_{ij}\Big |_{\widehat{T}_{\xo} M} \right\}_{1 \leq i,j \leq n, \,
    \ell({\mathbf{m}}) = p},
\end{equation}
and introduce the Lie algebra
\begin{equation}
  \label{eq:4}
  {\mathcal A}:= \oplus_{p=0}^\infty {\mathcal A}_p.
\end{equation}

\begin{remark}
  \label{rem:A}
  The elements of $\mathcal A$ can be seen as 1-homogeneous vector fields
  on $\widehat{T}_{\xo} M$, or equivalently, as vector fields on the
  indicatrix $\I_{\xo} \simeq \mathbb S^{n-1}$ with polynomial
  coefficients. Indeed, in \eqref{eq:4_p} the denominators of the
  coefficients of the curvature vector fields $\xi_{ij}$ is $\norm{y}^p$
  which is identically 1 on the indicatrix at the particular point
  $\xo \in M$.
\end{remark}

\smallskip

\begin{lemma}
  \label{lemma:A}
  The Lie algebra $\mathcal A$ satisfies
  $\mathcal A \subset \ihol_{\xo}(\F)$.
\end{lemma}

\begin{proof}
  The infinitesimal holonomy algebra $\ihol_{\xo}(\F)$ is generated by the
  curvature vector fields, their successive horizontal covariant
  derivatives, and contains their Lie brackets. We will show, using
  mathematical induction, that the generating elements of $\mathcal A_p$
  ($p\geq 0$) can be express as linear combination of elements of
  $\ihol_{\xo}(\F)$.

  \smallskip

  $\bullet$ $p=0$. The elements of $\mathcal A_0$ are just the linear
  combination of curvature vector fields $\xi_{ij}$ with constant
  coefficients, and by definitions they are elements of $\ihol_{\xo}(\F)$.
  
  \smallskip

  $\bullet$ $p=1$. From Lemma \ref{lemma:cov_R} we know, that the
  horizontal covariant derivative of the curvature tensor vanishes. It
  follows that for the curvature vector fields
  \begin{math}
    \xi_{ij} = R(\frac{\partial}{\partial x^i}, \frac{\partial}{\partial
      x^j}),
  \end{math}
  we get 
  \begin{equation}
    \label{eq:cov_R}
    \nabla_k \xi_{ij} = G^s_{ki}\xi_{sj} + G^s_{kj}\xi_{is}. 
  \end{equation}
  At the particular point $\xo\in M$, using the formula
  \eqref{eq:proj_flat_G_ijk} and \eqref{eq:curv_field}, one can obtain:
  \begin{equation}
    \label{eq:nabla_k}
    \nabla_k \xi_{ij}
    = \tfrac{c}{\norm{y}} \left( 2 y^k \xi_{ij}
      + \delta^k_i y^s \xi_{sj} 
      - \delta^k_j y^s \xi_{si} \right).
  \end{equation}
  As a consequence, for any pairwise different indeces $i,j,k$ we get
  \begin{equation}
    \label{eq:nabla_k_ij}
    \tfrac{y^k}{\norm{\mathbf{y}}} \xi_{ij}  = \tfrac{1}{2c}
    \nabla_k \xi_{ij} \quad \in
    \ihol_{\xo}(\F) \qquad i\neq j, \ i\neq
    k, \ j\neq k, 
  \end{equation}
  showing that \eqref{eq:nabla_k_ij} are in $\ihol_{\xo}(\F)$.  Moreover, for
  $k=i$ from \eqref{eq:nabla_k}, we get the linear system:
  \begin{equation}
    \label{eq:matrix}
    \hphantom{\qquad j=1,\dots, n,}
    \begin{pmatrix}
      \nabla_{1}\xi_{1j} \\[-3pt]
      \vdots \\[-3pt]
      \left[
        \nabla_{j}\xi_{jj}
      \right] \\[-3pt]
      \vdots \\[-3pt]
      \nabla_{n}\xi_{nj}
    \end{pmatrix}
    = c
    \begin{pmatrix}
      3 & 1 &   \dots  & 1 \\
      \vdots &  \ddots & & \vdots\\
      \vdots & &  \ddots &  \vdots \\
      1 & 1 &   \dots & 3 \\
    \end{pmatrix}
    \begin{pmatrix}
      \frac{y^1}{\norm{y}}\xi_{1j} \\[-5pt]
      \vdots \\[-5pt]
      \left[ \frac{y^j}{\norm{y}}\xi_{jj}
      \right] \\[-5pt]
      \vdots \\[-5pt]
      \frac{y^n}{\norm{y}}\xi_{nj} 
    \end{pmatrix},
    \qquad j=1,\dots, n,
  \end{equation}
  where in the column matrices the (trivially zero) terms written in square
  brackets are missing.  Since the quadratic matrix in \eqref{eq:matrix} is
  invertible, we get that $\frac{y^i}{\norm{y}}\xi_{ij}$ can be expressed
  as a linear combination of covariant derivatives of curvature vector
  fields, therefore they are also elements of $\ihol_{\xo}(M)$.  Therefore
  we obtained that the generating elements of $\mathcal A_1$ are in
  $\ihol_{\xo}(M)$, therefore $\mathcal A_1 \subset \ihol_{\xo}(M)$.
  \bigskip
  
  $\bullet$ $p=2$. The second covariant derivatives of the curvature vector
  fields are elements of $\ihol_{\xo}(M, \F)$. Calculating their expression
  we get
  \begin{equation}
    \label{eq:second_cov}
    \begin{aligned}
      \nabla_m \nabla_k \xi_{ij}
      = & (\lambda + c^2) (\delta^m_j \xi_{ki} - \delta^m_i \xi_{kj}) + c^2
      ( \delta^k_j \xi_{mi} - \delta^k_i \xi_{mj} ) + 2 (c^2-\lambda
      )\delta^m_k \xi_{ij}
      \\
      & + 4 \tfrac{c^2}{\norm{y}^2} \left(y^m y^k \xi_{ij} +
        \delta^i_k y^m y^s \xi_{sj} - \delta^j_k y^m y^s \xi_{si} -
        \delta^j_m y^k y^s \xi_{si} + \delta^i_m y^k y^s \xi_{sj}\right).
  \end{aligned}    
\end{equation}
As special cases we can get (different letters denote different indeces,
repeated indeces does not mean summation in the formula below):
\begin{subequations}
  \label{eq:second_cov_particular}
  \begin{align}
    \label{eq:second_cov_mkij}
    \nabla_m\nabla_k \xi_{ij}
    & = + 4 c^2\tfrac{y^m y^k}{\norm{y}^2}  \xi_{ij}
    \\
    \label{eq:second_cov_kkij}
    \nabla_k\nabla_k \xi_{ij}
    & = 2 (c^2-\lambda )   \xi_{ij} + 4c^2 \tfrac{(y^k)^2}{\norm{y}^2}  \xi_{ij},
    \\
    \label{eq:second_cov_iiij}
    \nabla_i\nabla_i \xi_{ij}
    & =-3 \lambda \xi_{ij}
      + 4c^2 \tfrac{(y^i)^2}{\norm{y}^2}  \xi_{ij},
      + 8c^2  \sum_{s=1}^n  \tfrac{y^i y^s}{\norm{y}^2} \xi_{sj} 
    \\[-7pt]		
    \label{eq:second_cov_ikij}
    \nabla_i \nabla_k \xi_{ij}
    & = 
      - (\lambda + c^2)  \xi_{kj}
      + 4 c^2 \tfrac{y^i y^k}{\norm{y}^2}  \xi_{ij}
      + 4 c^2  \sum_{s=1}^n \tfrac{y^k y^s}{\norm{y}^2}   \xi_{sj} 
  \end{align}    
\end{subequations}
From the equation \eqref{eq:second_cov_mkij} we obtain that
$\frac{y^k y^m}{\norm{y}^2} \xi_{ij}$ can be expressed as a constant
multiple of the second derivatives of curvature vector fields, therefore it
is in $\ihol_{\xo}(M, \F)$. Similarly, from equation
\eqref{eq:second_cov_kkij} we get that
$\frac{(y^k)^2}{\norm{y}^2} \xi_{ij}$ can be expressed with the combination
of the curvature vector fields and teir second covariant derivatives,
therefore they are elements of $\ihol_{\xo}(M, \F)$.  Equation
\eqref{eq:second_cov_ikij} can be considered as a linear system on
$\frac{ y^k y^i}{\norm{y}^2}\xi_{ij}$
\begin{equation}
  \label{eq:matrix1}
  \begin{pmatrix}
    \nabla_{1} \! \nabla_{k}\xi_{1j} \! + \! (c^2 \! + \! \lambda)
    \xi_{kj} \\[-3pt]
    \vdots \\[-3pt]
    [{\nabla_{j}\nabla_{k}\xi_{jj} \! + \! (c^2 \! + \! \lambda) \xi_{kj}}]
    \\[-3pt]
    \vdots \\[-3pt]
    \nabla_{n}\nabla_{k}\xi_{nj} \! + \! (c^2 \! + \! \lambda) \xi_{kj}
  \end{pmatrix}
  = 4 c^2 
  \begin{pmatrix}
    2 & 1 & \hspace{-5pt} \dots \hspace{-5pt}
    & 1 \\[3pt]
    1 & \ddots & \hspace{-10pt} & \vdots
    \\[-4pt]
    \vdots & & \hspace{-5pt} \ddots \hspace{-5pt} &1
    \\[3pt]
    1  &  \hspace{-5pt} \dots & 1 \hspace{-5pt} & 2
  \end{pmatrix}
  \begin{pmatrix}
    \frac{y^k y^1}{\norm{y}^2}\xi_{1j} \\[-3pt]
    \vdots \\[-3pt]
    \left[ {\frac{ y^k y^j}{\norm{y}^2}\xi_{jj}}
    \right] \\[-3pt]
    \vdots \\[-3pt]
    \frac{ y^k y^n}{\norm{y}^2}\xi_{nj}
  \end{pmatrix}
\end{equation}
where the elements on the left hand side are elements of the infinitesimal
holonomy algebra $\ihol_{\xo}(\F)$. The $(n \! - \! 1)\times (n \! - \! 1)$
matrix appearing in \eqref{eq:matrix1} is regular, therefore the terms
$\frac{ y^k y^i}{\norm{y}^2}\xi_{ij}$ can be expressed as linear
combination of elements of the left hand side which are elements of
$\ihol_{\xo}(\F)$. It follows that
$\frac{ y^k y^i}{\norm{y}^2}\xi_{ij} \in \ihol_{\xo}(\F)$.  Finally, using
\eqref{eq:second_cov_iiij}, one can express
$\frac{ y^i y^i}{\norm{y}^2}\xi_{ij}$ with the help of the other terms
appearing in those equations.  Since these are all in $\ihol_{\xo}(\F)$, we
can get that $\frac{ y^i y^i}{\norm{y}^2}\xi_{ij} \in \ihol_{\xo}(\F)$.
One can conclude that the elements generating $\mathcal A_2$ can be
expressed as a linear combination of elements in the infinitesimal holonomy
algebra, therefore $\mathcal A_2 \subset \ihol_{\xo}(\F)$.

\bigskip

\noindent
$\bullet$ Let us assume that $A_\ell \in \ihol_{\xo}(M, \F)$ for
$1\leq \ell \leq p$ and we will show that $A_{p+1} \in \ihol_{\xo}(M, \F)$.
First we observe that
\begin{equation}
  \label{eq:A_1}
  \sum_{s=1}^{n}\tfrac{y^s}{\norm{y}}\xi_{ks}
  = \frac{1}{\lambda} 
  \left(
    \sum_{s=1}^n \tfrac{y^s y^k}{\norm{y}}
    \tfrac{\partial}{\partial y^s}
    - \norm{y} \tfrac{\partial}{\partial y^k}
  \right) 
  = \frac{1}{\lambda} 
  \left(\tfrac{y^k}{\norm{y}} C 
    - \norm{y} \tfrac{\partial}{\partial y^k}  \right) 
  \in \mathcal A_1,
\end{equation}
where $C$ denotes the canonica Liouville (or radial) vector field.  Let us
consider for the multiindex $\mathbf{m}:=(m_1, \dots m_n)$ where
$\ell(\mathbf{m})=m_1 + \dots +m_n=p$ the vector field
\begin{equation}
  \label{eq:A_p}
  \frac{\mathbf{y}^{\mathbf{m}}}{\norm{\mathbf{y}}^{\ell({\mathbf{m}})}} \,
  \xi_{ij}
  =  \frac{y_{1}^{m_1}y_{2}^{m_2} \dots y_{n}^{m_n}}{\norm{y}^p}\xi_{ij}
  \qquad \in  \mathcal A_p.
\end{equation}
By the induction hypothesis $\mathcal A_1\subset \ihol_{\xo}(M)$ and
$\mathcal A_p\subset \ihol_{\xo}(M)$, therefore -- using the fact that
$\ihol_{\xo}(M)$ is a Lie algebra -- we get that
\begin{equation}
  \label{rem:bracket_hol}
  [\mathcal A_1, \mathcal A_p] \subset \ihol_{\xo}(M).
\end{equation}
In particular, the Lie bracket of the elements \eqref{eq:A_1} and
\eqref{eq:A_p} is an element of $\ihol_{\xo}(M)$. If $i,j,k$ are pairwise
different indeces, then one can obtain for the multiindex $\mathbf{m}$ of
length $\ell(\mathbf{m})=p$:
\begin{equation}
  \label{eq:bracket_1_p}
  \begin{aligned}
    \left[
      \frac{\mathbf{y}^{\mathbf{m}}}{\norm{\mathbf{y}}^{\ell({\mathbf{m}})}}
      \, \xi_{ij}, \sum_{s=1}^{n}\frac{y^s}{\norm{y}}\xi_{ks} \right]
    & = \left[ \frac{\mathbf{y}^{\mathbf{m}}}{\norm{\mathbf{y}}^{p}} \,
      \xi_{ij}, \tfrac{1}{\lambda} \Bigl( \tfrac{y^k}{\norm{y}} C -
      \norm{y} \tfrac{\partial}{\partial y_{k}} \Bigr)\right]
    = m_k
    \tfrac{\mathbf{y}^{\mathbf{m}-\mathbf{1_k}}}{\norm{\mathbf{y}}^{p-1}}
    \, \xi_{ij} - p
    \tfrac{\mathbf{y}^{\mathbf{m}+\mathbf{1_k}}}{\norm{\mathbf{y}}^{p+1}}
    \, \xi_{ij},
  \end{aligned}
\end{equation}
where $\mathbf{1_k}=(0, \dots, 1, \dots 0)$ denotes the multiindex having
$1$ at the $k$th position. From \eqref{rem:bracket_hol} the left hand side
of \eqref{eq:bracket_1_p} is in the infinitesimal holonomy, and
from the induction hypothesys
\begin{math}
  \frac{\mathbf{y}^{\mathbf{m}-\mathbf{1_k}}}{\norm{\mathbf{y}}^{p-1}} \,
  \xi_{ij} \in \mathcal A_{p-1},
\end{math}
is also an element of $\ihol_{\xo}(M)$. It follows that
\begin{equation}
  \label{eq:A_p_kij}
  \hphantom{i\neq j, \  k\neq i, \ k\neq j}
  \frac{\mathbf{y}^{\mathbf{m}+\mathbf{1_k}}}{\norm{\mathbf{y}}^{p+1}} \,
  \xi_{ij} \in \ihol_{\xo}(\F), \qquad i\neq j, \  k\neq i, \ k\neq j.   
\end{equation}
Similarly, 
\begin{equation}
  \label{eq:bracket_2_p}
  \begin{aligned}
    \left[
      \frac{\mathbf{y}^{\mathbf{m}}}{\norm{\mathbf{y}}^{\ell({\mathbf{m}})}}
      \, \xi_{ij}, \sum_{s=1}^{n} \frac{y^s}{\norm{y}}\xi_{is} \right]
    = m_i
    \frac{\mathbf{y}^{\mathbf{m}-\mathbf{1}_i}}{\norm{\mathbf{y}}^{p-1}}
    \, \xi_{ij}
    +(1-p)
    \frac{\mathbf{y}^{\mathbf{m}+\mathbf{1}_i}}{\norm{\mathbf{y}}^{p+1}}
    \, \xi_{ij}
    + \sum_{s=1, s \neq i}^n
    \frac{\mathbf{y}^{\mathbf{m}+\mathbf{1}_s}}{\norm{\mathbf{y}}^{p+1}}
    \, \xi_{sj}.
  \end{aligned}
\end{equation}
Using \eqref{eq:A_p_kij} we obtain that
\begin{equation}
  \label{eq:A_p_iij}
  \hphantom{\qquad i \neq j}
  \frac{\mathbf{y}^{\mathbf{m}+\mathbf{1}_i}}{\norm{\mathbf{y}}^{p+1}} \,
  \xi_{ij} \in \ihol_{\xo}(\F), \qquad i \neq j. 
\end{equation}
From \eqref{eq:A_p_kij} and \eqref{eq:A_p_iij} we get that
$\mathcal A_{p+1} \subset \ihol_{\xo}(\F)$ which completes the proof of Lemma
\ref{lemma:A}.

\end{proof}

\medskip

\begin{proposition}
  \label{prop:ihol_dense}
  Let $(M, \F)$ be a projectively flat spherically symmetric Finsler
  manifold of constant curvature $\lambda \neq 0$ and $\xo$ a point where
  \eqref{eq:multiple_norm} is satisfied.  Then the infinitesimal holonomy
  algebra $\ihol_{\xo}(\F)$ is dense in $\X{\I_{\xo}}$.
\end{proposition}

\begin{proof}
  Let $(M, \F)$ be a projectively flat spherically symmetric Finsler
  manifold of constant curvature $\lambda \neq 0$ and $\xo$ a point where
  \eqref{eq:multiple_norm} is satisfied.  According to Remark \ref{rem:A},
  the elements of $\mathcal A$ can be considered as vector fileds on the
  indicatrix $\I_{\xo} \simeq \mathbb S^{n-1}$ with polynomial
  coefficients that is   
  \begin{equation}
    \label{eq:AS}
    \mathcal A = \halmazvonal{Q^{ij} \xi_{ij}}{Q^{ij}\in
      \pols},
  \end{equation}
  where
  \begin{equation}
    \label{eq:polynomials_S}
    \pols = \R[y^1, \dots, y^n] \big |_{\mathbb S^{n-1}}
    = 
    \left\{ p \, \big |_{_{\mathbb S^{n-1}}} \  \Big | \ p \in
      \mathbb R[y^1, \dots, y^n] \right\} ,
  \end{equation}
  is the algebra of polynomial functions on $\mathbb S^{n-1} \subset \R^n$.  

  Nachbin's theorem \cite{Nachbin_1949, Nachbin_1979} gives an analog for
  Stone--Weierstrass theorem for algebras of real values $C^k$ functions on
  a $C^k$ manifold, $k=1, \dots, \infty$. It states that if $\mathcal S$ is
  a subalgebra of the algebra of $C^k$ smooth functions on a finite
  dimensional $C^k$ smooth manifold $M$, and $\mathcal S$ separates the
  points of $M$ and also separates the tangent vectors of $M$ in the sense
  that for each point $x \in M$ and tangent vector $v \in T_xM$, there is
  an $f \in \mathcal S$ such that $df_x(v) \neq 0$, then $\mathcal S$ is
  dense in $C^k(M)$. Clearly, $\pols$, the algebra of polynomial functions
  on $\mathbb S^{n-1}$ satisfies Nachbin's conditions, therefore $\pols$ is
  dense in $C^\infty(\mathbb S^{n-1})$ with respect to the $C^\infty$
  topology:
  \begin{equation}
    \label{eq:pol_sphere}
    \overline{\pols} = C^\infty(\mathbb S^{n-1}).
  \end{equation}
  On the other hand, any vector field $X\in \X{\mathbb S^{n-1}}$ can be
  written as
  \begin{math}
    \label{eq:vector_on_S}
    X = X^{ij} \xi_{ij},
  \end{math}
  where $X^{ij}$ are smooth function on $\mathbb S^{n-1}$, and $\xi_{ij}$
  are the infinitesimal generator of rotations which are the curvature
  vector fields \eqref{eq:curv_field}, that is
  \begin{equation}
    \label{eq:XS}
    \X{\mathbb S^{n-1}}
    = \halmazvonal{X^{ij} \xi_{ij}}{X^{ij}\in C^\infty(\mathbb S^{n-1})},
  \end{equation}
  and from \eqref{eq:pol_sphere} we get that \eqref{eq:AS} is a dense
  subset in \eqref{eq:XS} with respect to the $C^k$ topology, that is
  \begin{equation}
    \label{eq:P_S}
    \overline{\, \mathcal A \, }= \mathfrak X(\mathbb S^{n-1}).
  \end{equation}
  From Lemma \ref{lemma:A} we get
  \begin{equation}
    \label{eq:AiholX}
    \mathcal A \ \subset \ \ihol_{\xo}(\F) \ \subset \ \X{\mathbb S^{n-1}},
  \end{equation}
  it follows 
  \begin{equation}
    \label{eq:clAiholX}
    \overline{\,  \mathcal A \, }
    \ \subset \ \overline{\ihol_{\xo}(\F)}
    \ \subset \ \X{\mathbb S^{n-1}},
  \end{equation}
  and from \eqref{eq:P_S} one can obtain
  \begin{equation}
    \label{eq:closure_iholS}
    \overline{\, \ihol_{\xo}(\F) \, } = \X{\mathbb S^{n-1}}
  \end{equation}
  that is $\ihol_{\xo}(\F)$ is dense in $\mathfrak X(\mathbb S^{n-1})$ with
  respect to the $C^\infty$ topology.
  
\end{proof}

\bigskip

\begin{theorem}
  \label{thm:2_hol_group_max}
  Let $(M, \F)$ be a simply connected, projectively flat spherically
  symmetric Finsler manifold of constant curvature $\lambda \neq 0$ and
  $\xo$ a point where \eqref{eq:multiple_norm} is satisfied.  Then the
  holonomy group $\Hol_{\xo}(\F)$ is maximal, that is its closure is
  isomorphic to $\diffo{\mathbb S^{n-1}}$, the connected component of the
  identity of the group of smooth diffeomorphism on the ${n-1}$-dimensional
  sphere.
\end{theorem}

\begin{proof}
  The proof is a direct consequence of Theorem \ref{thm:1_ihol} and
  Proposition \ref{prop:ihol_dense}. Indeed, let $(M, \F)$ be a projectively
  flat spherically symmetric Finsler manifold of constant curvature
  $\lambda \neq 0$ and $\xo$ a point where \eqref{eq:multiple_norm} is
  satisfied. From Proposition \ref{prop:ihol_dense} we get that the
  infinitesimal holonomy algebra $ hol_{\xo}^{*}(\F)$ is dense in
  $\X{\I_{\xo}}$, and from Theorem \ref{thm:1_ihol} we get that in that
  case the holonomy group is maxmial
  \begin{equation}
    \label{eq:closure_Hol_2}
    \overline{\Hol_x(\F)} \cong \diffo{\mathbb S^{n-1}},
  \end{equation}
  that is its closure is isomorphic to the connected component of the
  identity in the diffeomorphism group of the $n-1$-dimensional sphere.  

\end{proof}

\bigskip
\bigskip

\begin{example} 
  \label{example:funk} 
  (P. Funk, \cite{Funk_1929,Funk_1963}) The \emph{standard Funk
    manifold} $(\mathbb D^n, \F)$ defined by the metric function
  \begin{equation}
    \label{projective1} \F(x,y) = \frac{\sqrt{|y|^2 - \left(|x|^2|y|^2 - 
          \langle x,y\rangle^2\right)}}{1 - |x|^2}\pm\frac{\langle x,y\rangle}
    {1 - |x|^2} 
  \end{equation}
  on the unit open disk $\mathbb D^n \subset \mathbb R^n$ is projectively
  flat with constant flag curvature $\lambda=-\frac14$. Its projective
  factor can be computed using formula (\ref{eq:P}):
  \begin{equation}
    \label{projective2}
    \P(x,y) =  \frac12 \;\frac{\pm\sqrt{|y|^2 - \left(|x|^2|y|^2  
          - \langle x,y\rangle^2\right)} + \langle x,y\rangle}{1 - |x|^2}.
  \end{equation}
  At $\xo=(0,\dots, 0) \in \mathbb D^n$ we have $\F(\xo,y)=\left|y\right|$
  and $\P(x_0,y)=\pm\frac{1}{2}\left|y\right|$.  Thus, at $\xo$ the
  standard Funk metric satisfies the condition of Theorem
  \ref{thm:2_hol_group_max}, therefore its holonomy group $\Hol_{\xo}(\F)$
  is maximal and isomorphic to $\diffo{\mathbb S^{n-1}}$.
\end{example}

\begin{example} 
  \label{example:bryant_shen} 
  The \emph{Bryant--Shen metric} $\F_\alpha$, ($|\alpha|<\frac{\pi}{2}$),
  are the elements of a $1$-parameter family of projective Finsler metric
  with constant flag curvature $\lambda=1$:
  \begin{equation}
    \label{eq:bryant_shen}
    \F_\alpha=\sqrt{\frac{\sqrt{A}+B}{2D}+
      \left(
        \frac{C}{D}
      \right)^2}+\frac{C}{D},
    \qquad
    \P_\alpha= -\sqrt{\frac{\sqrt{A}-B}{2D}- 
      \left(\frac{C}{D}\right)^2}-\frac{\widetilde{C}}{D},
  \end{equation}
  where 
  \begin{equation}
    \begin{aligned}
      A&=\big(\cos(2\alpha)\left| y \right|^2 + \norm{x}^2 \! \norm{y}^2-
      \langle x,y \rangle^2 \big)^2+ \bigl( \sin(2\alpha)\norm{y}^2
      \bigr)^2,
      \\
      B& =\cos(2\alpha)\left| y \right|^2 + \norm{x}^2 \! \norm{y}^2-
      \langle x,y \rangle^2,
      \\
      C&=\sin(2\alpha)\langle x,y \rangle,
      \\
      \widetilde{C}&= \big(\cos(2\alpha)+\left| x \right|^2 \big) \langle
      x,y \rangle,
      \\
      D&=\left| x \right|^4 +2 \cos(2\alpha)\left| x \right|^2+1.
  \end{aligned}
\end{equation}
The norm function and the projective factor at $0\!\in\!  \mathbb R^n$ have
the form
\begin{equation}\label{BrSh}
  \F_\alpha(0,y) = |y| \,\cos \alpha, \qquad 
  \P_\alpha (0,y) = |y|\,\sin \alpha, \qquad \norm{\alpha}<\frac{\pi}{2},
  \end{equation}
  in a local coordinate system corresponding to the Euclidean canonical
  coordinates, centered at $0\in \mathbb R^n$. R.~Bryant in
  \cite{Bryant_1996, Bryant_1997} introduced and studied this class of
  Finsler metrics on $\mathbb S^2$ where great circles are
  geodesics. Z. Shen generalized its construction and obtained the
  expression \eqref{eq:bryant_shen} in \cite[Example 7.1.]{Shen_2003b}.
  Since the condition of Theorem \ref{thm:2_hol_group_max}, therefore the
  holonomy group $\Hol_{\xo}(\F_{\alpha})$ of the Bryant--Shen metric is
  maximal and isomorphic to $\diffo{\mathbb S^{n-1}}$.
\end{example}

\end{document}